\documentclass[12pt,reqno]{amsart}
\usepackage{amssymb,amsmath,amsthm,enumerate,bbm}
\usepackage[a4paper]{geometry}

\DeclareMathOperator{\sign}{sign}

\DeclareMathOperator{\Dom}{Dom}

\DeclareMathOperator{\Lip}{Lip}

\DeclareMathOperator{\sgn}{sgn}

\DeclareMathOperator{\supp}{supp}
\DeclareMathOperator{\dist}{dist}

\DeclareMathOperator{\clos}{clos}

\DeclareMathOperator{\BMO}{BMO}
\DeclareMathOperator{\VMO}{VMO}
\DeclareMathOperator{\CMO}{CMO}

\DeclareMathOperator{\Smooth}{Smooth}

\DeclareMathOperator{\const}{const}
\newcommand{\comp}{\text{\rm comp}}

\renewcommand\Im{\text{\rm Im}\,}

\newcommand{\abs}[1]{\lvert#1\rvert}

\newcommand{\norm}[1]{\lVert#1\rVert}
\newcommand{\Norm}[1]{\left\lVert#1\right\rVert}
\newcommand{\jap}[1]{\langle#1\rangle}

\newcommand{\bbR}{{\mathbb R}}
\newcommand{\bbC}{{\mathbb C}}
\newcommand{\bbN}{{\mathbb N}}
\newcommand{\bbZ}{{\mathbb Z}}

\newcommand{\wh}{\widehat}

\newcommand{\calH}{{\mathcal H}}
\newcommand{\calK}{{\mathcal K}}

\newcommand{\calB}{\mathcal{B}}

\newcommand{\Sch}{\mathbf{S}}

\DeclareFontFamily{U}{mathx}{\hyphenchar\font45}
\DeclareFontShape{U}{mathx}{m}{n}{<5> <6> <7> <8> <9> <10>
<10.95> <12> <14.4> <17.28> <20.74> <24.88> mathx10}{}
\DeclareSymbolFont{mathx}{U}{mathx}{m}{n}
\DeclareFontSubstitution{U}{mathx}{m}{n}
\DeclareMathAccent{\widecheck}{0}{mathx}{"71}

\numberwithin{equation}{section}
\renewcommand{\[}{\begin{equation}}
\renewcommand{\]}{\end{equation}}

\theoremstyle{plain}
\newtheorem{theorem}{\bf Theorem}[section]
\newtheorem*{theorem*}{Theorem 1.1$'$}
\newtheorem{lemma}[theorem]{\bf Lemma}
\newtheorem{proposition}[theorem]{\bf Proposition}

\theoremstyle{definition}

\newtheorem*{definition*}{\bf Definition}

\theoremstyle{remark}
\newtheorem*{remark*}{\bf Remark}

\newtheorem*{example*}{\bf Example}

\newcommand{\1}{\mathbbm{1}}

\DeclareFontFamily{U}{mathx}{\hyphenchar\font45}
\DeclareFontShape{U}{mathx}{m}{n}{<5> <6> <7> <8> <9> <10>
<10.95> <12> <14.4> <17.28> <20.74> <24.88> mathx10}{}
\DeclareSymbolFont{mathx}{U}{mathx}{m}{n}
\DeclareFontSubstitution{U}{mathx}{m}{n}
\DeclareMathAccent{\widecheck}{0}{mathx}{"71}

\date{4 July 2019}

\title[Schatten conditions for  Schr\"odinger operators]{Schatten class conditions for functions of Schr\"odinger operators}

\author{Rupert L. Frank}
\address[Rupert L. Frank]{Mathematisches Institut, Ludwig-Maximilans Universit\"at M\"unchen, Theresienstr. 39, 80333 M\"unchen, Germany, and Department of Mathematics, California Institute of Technology, Pasadena, CA 91125, USA}
\email{rlfrank@caltech.edu}

\author{Alexander Pushnitski}
\address[Alexander Pushnitski]{Department of Mathematics, King's College London, Strand, London, WC2R 2LS, UK}
\email{alexander.pushnitski@kcl.ac.uk}

\begin{document}

\begin{abstract}
We consider the difference $f(H_1)-f(H_0)$, where $H_0=-\Delta$ and $H_1=-\Delta+V$ are the 
free and the perturbed Schr\"odinger operators in $L^2(\bbR^d)$, and $V$ is a real-valued short range potential. 
We give a sufficient condition for this difference to belong to a given Schatten class $\Sch_p$, depending
on the rate of decay of the potential and on the smoothness of $f$ (stated in terms of the membership
in a Besov class).  In particular, for $p>1$ we allow 
for some unbounded functions~$f$. 
\end{abstract}

\maketitle

\section{Introduction and main results}\label{sec.a}

\subsection{Overview}

Let $H_0$ and $H_1$ be the free and the perturbed (self-adjoint) Schr\"odinger operators, 
\[
H_0=-\Delta, \quad H_1=-\Delta+V
\quad\text{ in $L^2(\bbR^d)$, $d\geq1$,}
\label{a00}
\]
where the real-valued potential $V$ satisfies the bound
\[
\abs{V(x)}\leq C(1+\abs{x})^{-\rho}, \quad \rho>1.
\label{a0}
\]
The purpose of this paper is to give new sufficient conditions for the boundedness and 
the Schatten class membership of the difference
$$
D(f):=f(H_1)-f(H_0)
$$
where $f$ is a complex-valued function on $\bbR$ of an appropriate class. 
These conditions are given in terms of the smoothness of $f$ and the  
exponent $\rho$ in \eqref{a0}. 
This paper is a continuation of \cite{I}, where this problem was considered in the 
general operator theoretic context. 
It is also a further development of \cite{FP}, where the trace class membership of
$D(f)$ was considered. As explained in \cite{FP} and briefly recalled in Subsection \ref{sec:motivation} below, this problem is in part motivated by applications to mathematical physics. 

As it is well known, the continuous spectrum of both $H_0$ and $H_1$ 
consists of the closed positive half-line $[0,\infty)$. 
We focus on the local behaviour of $f$ on $(0,\infty)$. 
The questions of the behaviour of $f$ at $+\infty$ and near zero are of a very different nature, 
so in what follows we assume that $f$ is compactly supported on $(0,\infty)$. As explained in Subsection \ref{sec:motivation}, this is not a severe restriction in the applications that we have in mind.

If $f$ is sufficiently smooth, say, $f\in C_0^\infty(0,\infty)$, and the exponent $\rho$ 
is sufficiently large, then it is not difficult to show, by a variety of standard methods, 
that the difference $D(f)$ is trace class.
On the other hand, as shown in \cite{PYa3}, 
if $f$ has a jump discontinuity at a point $\lambda>0$,
then $D(f)$ is never compact, unless scattering at energy $\lambda$ is trivial. 
Thus, a question arises how the transition from the non-compact to the compact difference $D(f)$ occurs
when the smoothness of $f$ increases. 
The ``degree of compactness" of $D(f)$ will be measured by its Schatten class membership, 
and the ``degree of smoothness" of $f$ --- by its Besov class membership.

Our key example is of $f$ having an isolated cusp-like singularity 
(see \eqref{a1a}, \eqref{a0a} below) on the positive half-line,  smooth elsewhere and compactly supported.

\subsection{Boundedness and compactness of $D(f)$}

Below $\BMO(\bbR)$ is the class of functions  of bounded mean oscillation on $\bbR$, 
and $\VMO(\bbR)$ (vanishing mean oscillation) is the closure of $C(\bbR)\cap\BMO(\bbR)$ 
in $\BMO$. 
Further, $\calB$ and $\Sch_\infty$ are the classes of bounded and compact operators on $L^2(\bbR^d)$. 
Precise definitions are given in Section~\ref{sec.b}.

\begin{theorem}\label{thm.a3alt}
Let $H_0$, $H_1$ be as in \eqref{a00}, \eqref{a0} with $\rho>1$.
\begin{enumerate}[\rm (i)]
\item
For any $f\in \BMO(\bbR)$ with compact support in $(0,\infty)$, we have $D(f)\in\calB$. 
\item
For any $f\in \VMO(\bbR)$ with compact support in $(0,\infty)$, we have $D(f)\in\Sch_\infty$. 
\end{enumerate}
\end{theorem}

To illustrate the type of admissible singularities for the function $f$ in the above theorem, let us 
consider the following example.  
Let $\chi_0\in C_0^\infty(\bbR)$ be a function which equals $1$ in a neighbourhood of the origin
and vanishes outside the interval $(-c,c)$ with some $0<c<1$. 
Then the function 
\[
f(x)=\chi_0(x)\abs{\log\abs{x}}
\label{a1a}
\]
is in $\BMO(\bbR)$, and the function 
$$
f_\gamma(x)=\chi_0(x)\abs{\log\abs{x}}^\gamma
$$
is in $\VMO(\bbR)$ if $\gamma<1$. 
Of course, the same applies to all shifted functions $f(x-\lambda)$, $f_\gamma(x-\lambda)$ for $\lambda\in\bbR$. 
Observe that these functions are unbounded for $\gamma>0$; 
this is perhaps the most striking feature of Theorem~\ref{thm.a3alt}. 
Observe also that functions with a jump discontinuity are in $\BMO$, but not in $\VMO$.

\subsection{Schatten class membership of $D(f)$}

For $0<p<\infty$, $B_{p,p}^{1/p}(\bbR)$ is the Besov class of functions on $\bbR$ 
and $\Sch_p$ is the Schatten class of all compact
operators in  $L^2(\bbR^d)$; see Section~\ref{sec.b}.

\begin{theorem}\label{thm.a3}
Let $H_0$, $H_1$ be as in \eqref{a00}, \eqref{a0}. 
\begin{enumerate}[\rm (i)]
\item
Assume $1<\rho\leq d$. 
Then for any $p>\dfrac{d-1}{\rho-1}$ and for any $f\in B_{p,p}^{1/p}(\bbR)$
with compact support in $(0,\infty)$, we have $D(f)\in\Sch_p$. 
\item
Assume $\rho>d$. Then for any $p>d/\rho$ and for any $f\in B_{p,p}^{1/p}(\bbR)$
with compact support in $(0,\infty)$, we have $D(f)\in\Sch_p$. 
\end{enumerate}
\end{theorem}

For $p=1$, this is the main result of \cite{FP}.

To illustrate the type of local singularities allowed for the functions $f\in B_{p,p}^{1/p}(\bbR)$, 
consider the following example. 
Let $\chi_0\in C_0^\infty(\bbR)$ be as above; 
fix $\alpha>-1$, $a_+,a_-\in\bbC$, and consider the function 
\[
F_\alpha(x)=
\begin{cases}
a_+\chi_0(x)\abs{\log\abs{x}}^{-\alpha}, & x>0,
\\
a_-\chi_0(x)\abs{\log\abs{x}}^{-\alpha}, & x<0.
\end{cases}
\label{a0a}
\]
It can be shown that (see \cite{PYa} or \cite[Proposition 1.3]{I})
\begin{enumerate}[(i)]
\item
If $a_+\not=a_-$ and $\alpha>0$, then $F_\alpha\in B_{p,p}^{1/p}(\bbR)$
if and only if $p>1/\alpha$. 
\item
If $a_+=a_-\not=0$ and $\alpha>-1$, then $F_\alpha\in B_{p,p}^{1/p}(\bbR)$ if and only if $p>1/(\alpha+1)$.
\end{enumerate}
We see that for $p>1$, the functions $F_\alpha$ may be unbounded. 
On the other hand, for $0<p\leq1$, the functions in $B_{p,p}^{1/p}(\bbR)$ are always bounded and continuous.

\subsection{Discussion}\label{sec:discussion}
Prior to our work \cite{FP}, the sharpest sufficient conditions for 
Schatten class inclusions for $D(f)$ were obtained through general operator theoretic estimates
of the form \cite{PotapovSukochev}
\[
\norm{f(H_1)-f(H_0)}_p\leq C(p)\norm{f}_{\Lip(\bbR)}\norm{H_1-H_0}_{p},
\quad
1<p<\infty,
\label{aa1}
\]
with appropriate modifications for $p=1$ and $p=\infty$; see \cite{Peller1}. 
Here $\Lip(\bbR)$ is the Lipschitz class and $\norm{\cdot}_p$ is the norm in $\Sch_p$. 
Of course, for the Schr\"odinger operator, the difference $V=H_1-H_0$ is never in $\Sch_p$, 
but one can apply \eqref{aa1} to the resolvents of $H_0$, $H_1$ or their powers. 

Observe that none of the functions \eqref{a1a}, \eqref{a0a} is in $\Lip(\bbR)$ (unless $\alpha=0$); 
they are not even in any H\"older class. So one cannot hope to deduce Theorem~\ref{thm.a3} 
from~\eqref{aa1}. 

In \cite{FP}, we have used an ad hoc calculation, 
combining Kato smoothness with an integral representation for $B_{1,1}^1$ functions
to prove Theorem~\ref{thm.a3} for $p=1$. 
In \cite{I} we approach the problem in a more systematic fashion;
working in a general operator theoretic framework, 
we introduce the notion of $\Sch_p$-valued Kato smoothness 
and combine it with the double operator integral technique of Birman and Solomyak to 
treat all cases $0<p<\infty$; see Sections~\ref{sec.b4} and \ref{sec.b5} below. 
Here we apply and adapt the general results of \cite{I} to the Schr\"odinger operators $H_0$, $H_1$. 

We emphasize that while the arguments in the present paper are much more special than the theory developed in \cite{I}, they are by no means restricted to the case where the unperturbed operator is the Laplacian. Rather, the basic underlying assumption is that the unperturbed operator has a `nice' diagonalization in an interval containing the support of the function $f$ and that its resolvent, or powers thereof, satisfy some trace ideal properties when multiplied by decaying functions. For instance, our results should remain valid when $-\Delta$ is replaced by $-\Delta+V_0(x)$ where $V_0$ is periodic and the function $f$ is supported away from band edges. Other examples are the three dimensional Landau Hamiltonian (with $f$ supported away from the Landau levels) or the Stark operator. In these cases the function $(1+|x|)^{-\rho}$ in \eqref{a0} needs to be modified appropriately. Yet another example is the discrete Laplacian. We omit the details, but refer to Section 11 of \cite{Pu} for some of the necessary ingredients for these extensions in some cases.

Another generalization that we do not pursue here is to replace the pointwise assumption \eqref{a0} on $V$ by an integral assumption. In \cite{FP} we showed that this was possible for $p=1$.

\subsection{Some ideas of the proof}

To prove our main results we proceed as follows. 
Let $\Lambda$  be an open bounded interval in $\bbR$, such that 
$\supp f\subset \Lambda$ and 
the closure of $\Lambda$ is included in $(0,\infty)$. 
We denote by $\1_\Lambda$ (resp. by $\1_{\Lambda^c}$) 
the characteristic function of $\Lambda$ (resp. of the complement $\Lambda^c$)  in $\bbR$. 
We write
\begin{align}
D(f)& =(\1_{\Lambda}(H_1)+\1_{\Lambda^c}(H_1))D(f)(\1_{\Lambda}(H_0)+\1_{\Lambda^c}(H_0))
\notag \\
& =
\1_{\Lambda}(H_1)D(f)\1_{\Lambda}(H_0)
-\1_{\Lambda^c}(H_1)f(H_0)+f(H_1)\1_{\Lambda^c}(H_0);
\label{f4}
\end{align}
here several terms vanish because of the assumption $\supp f\subset \Lambda$. 
We estimate the ``diagonal term" $\1_{\Lambda}(H_1)D(f)\1_{\Lambda}(H_0)$ by directly applying the results of \cite{I}
and some variants of the limiting absorption principle. 
We estimate the ``off-diagonal terms" (the second and third terms in the right side of \eqref{f4})
by using rather standard Schatten class bounds for Schr\"odinger operators. 

Following the proofs, it is not difficult to obtain estimates for the relevant norms of $D(f)$ in terms of
the exponents $p$, $\rho$, $d$, and the geometry of the support of $f$. However, these
estimates are clearly very far from being optimal (perhaps with the exception of the ones for the diagonal term
in \eqref{f4} above), and so we have not attempted to work them out explicitly. 

\subsection{Motivations from mathematical physics}\label{sec:motivation}

In a number of problems from mathematical physics one encounters differences $f(H_1)-f(H_0)$ where $H_1$ and $H_0$ are Schr\"odinger operators as in \eqref{a00} (or their generalizations mentioned in Subsection \ref{sec:discussion}) and where either the function $f$ is non-smooth at a certain $\mu>0$ or where the function $f$ belongs to a family of functions whose smoothness at a point $\mu>0$ degenerates in an asymptotic regime. While in these applications bounds on $f(H_1)-f(H_0)$ are needed most frequently in trace class norm, bounds in other Schatten norms or in operator norm are often a useful tool in the proofs.

We believe that our theorems and the methods we use to prove them are relevant in several such problems. The fact that our theorems are only stated for functions with compact support in $(0,\infty)$ is not a severe restriction since in many applications one can decompose $f=f_1+f_2$ where $f_1$ has compact support in $(0,\infty)$ and where $f_2$ is smooth. The contribution of $f_2$ to the difference can be controlled by \eqref{f4} or other standard bounds, while our theorems apply to $f_1$, which in the situations we have in mind gives the main contribution.

To be more specific, the function $f(x) = - \min\{x-\mu,0\}$ with $\mu>0$ appears in the problem of estimating the energy cost of making a hole in the Fermi sea. This cost was quantified through a version of the Lieb--Thirring inequality at positive density \cite{FLLS0,FLLS}. In order to convert the `density version' of this inequality into its `potential version', one needs the a priori information that $f(H_1)-f(H_0)$ is trace class. This was shown in \cite{FP} and is one of the basic motivations of this and our previous work \cite{I}. We emphasize that the above function $f$ does not satisfy the sufficient conditions from \cite{Peller1} which guarantee membership in the trace class.

The case where a family of smooth functions $f$ approaches a discontinuous function is relevant in the study of what is known as the Anderson orthogonality catastrophe; see \cite{GeKuMu,FP2} and references therein. The discontinuous limiting function is $f(x) = \chi_{\{x<\mu\}}$, while the functions approximating this function can be chosen smooth; see Section 3 in \cite{GeKuMu}. To be more precise, in this problem the product of $f(H_1)$ and $f(H_0)$ rather than their difference appears, but a mathematically closely related problem for the difference was studied by one of us in \cite{Pu}. In fact, in view of the latter work we believe that for both the operator norm and the Schatten norm with any fixed $0<p<\infty$ the assumptions on $\rho$ and $f$ in Theorems \ref{thm.a3alt} and~\ref{thm.a3} are best possible. Investigating this optimality, however, is beyond the scope of the present paper.

Different, but not unrelated bounds are relevant in the study of the entanglement entropy in quantum systems. We refer to \cite{LeSoSp,So} and references therein.


\subsection{The structure of the paper}
The paper can be divided into two parts: in Sections~\ref{sec.b}--\ref{sec.f}, 
we work in a general operator theoretic framework, and in Sections~\ref{sec.g}--\ref{sec.h} 
we specialise to the case of the Schr\"odinger operator. 

In Section~\ref{sec.b} we recall definitions of relevant function and operator classes, 
discuss the notions of Kato smoothness and $\Sch_p$-valued Kato smoothness and 
recall the main results of \cite{I}, which apply to estimates for the diagonal terms in \eqref{f4}. 
In Section~\ref{sec.f}, we prove preliminary estimates for the off-diagonal terms in \eqref{f4}. 

In Section~\ref{sec.g} we give sufficient conditions for $\Sch_p$-valued smoothness 
in the context of the Schr\"odinger operator. 
In Section~\ref{sec.5} we prove that certain auxiliary operators belong to relevant 
$\Sch_p$ classes; these facts are needed to treat the off-diagonal terms in \eqref{f4}. 
Finally, in Section~\ref{sec.h} we put everything together and prove Theorems~\ref{thm.a3alt} and \ref{thm.a3}.

\subsection*{Acknowledgements.} 
Partial support by U.S. National Science Foundation DMS-1363432 (R.L.F.) is acknowledged.
A.P. is grateful to Caltech for hospitality.

\section{Preliminaries}\label{sec.b}

\subsection{The classes $\BMO$ and $\VMO$}

The space $\BMO(\bbR)$ (bounded mean oscillation) consists of all locally integrable functions $f$ on $\bbR$
such that the following supremum over all bounded intervals $I\subset \bbR$
is finite:
\[
\sup_{I}\jap{\abs{f-\jap{f}_I}}_I<\infty, 
\quad
\jap{f}_I=\abs{I}^{-1}\int_I f(x)dx. 
\label{b10aa}
\]
Observe that this supremum vanishes on constant functions. 
Strictly speaking, the elements of 
$\BMO(\bbR)$ should be regarded not as functions but as
equivalence classes $\{f+\const\}$. However, since here we are interested in 
compactly supported functions $f$, this issue is not important to us. 
Functions in $\BMO(\bbR)$ belong to $L^p(-R,R)$ for any $R>0$ and any $p<\infty$, 
but not for $p=\infty$: they may have logarithmic singularities, see \eqref{a1a}.

Many explicit equivalent norms on $\BMO(\bbR)$ are known (see e.g. \cite{Girela}). 
The easiest one to define is the supremum in \eqref{b10aa}. 
In \cite{I} we use the norm related to Fefferman's duality theorem, which 
identifies $\BMO(\bbR)$ with the dual to the Hardy class $H^1$. 
This choice of the norm allowed us to explicitly determine the optimal constant appearing 
in the right hand side of \eqref{a4}. 
However, in this paper we do not attempt to keep track of all constants appearing in estimates, and
so the choice of the norm in $\BMO(\bbR)$ is not important here.

The subspace $\VMO(\bbR)\subset\BMO(\bbR)$ is characterised by the condition 
$$
\lim_{\epsilon\to 0} \sup_{|I|\leq\epsilon}\jap{\abs{f-\jap{f}_I}}_I = 0.
$$
Alternatively, $\VMO(\bbR)$ is the closure of $C(\bbR)\cap \BMO(\bbR)$ in $\BMO(\bbR)$. 

In \cite{I}, we also use the space $\CMO(\bbR)$ (continuous mean oscillation) which
can be characterised as the closure of $C_\comp(\bbR)\cap\BMO(\bbR)$ in $\BMO(\bbR)$. 
However, for a compactly supported function $f$, conditions $f\in\VMO$ and $f\in\CMO$ 
coincide. 

\subsection{The Besov class $B_{p,p}^{1/p}$}
Let $w\in C_0^\infty(\bbR)$, $w\geq0$, be a function with $\supp w\subset [1/2,2]$ and such that
$$
\sum_{j\in\bbZ}w_j(x)=1, \quad x>0, \quad \text{ where }w_j(x)=w(x/2^j). 
$$
The (homogeneous) Besov class $B_{p,p}^{1/p}(\bbR)$ is defined as the space of 
tempered distributions $f$ on $\bbR$ such that
\[
\norm{f}^p_{B_{p,p}^{1/p}}:=
\sum_{j\in\bbZ} 2^j 
\bigl(\norm{f*\wh w_j}_{L^p(\bbR)}^p+\norm{f*\overline{\wh w_j}}_{L^p(\bbR)}^p\bigr)
<\infty.
\label{b9}
\]
Here $\wh w_j$ is the Fourier transform of $w_j$, and $*$ is the convolution. 

We will only be interested in compactly supported elements in $B_{p,p}^{1/p}(\bbR)$. 
For compactly supported functions $f$, sufficient conditions for Besov class membership can be given
in terms of the usual Sobolev spaces: 
$$
f\in W_p^s(\bbR) \Rightarrow f\in B_{p,p}^{1/p}(\bbR), \quad s>1/p.
$$
(For $p\geq2$, this follows from \cite[Theorem 6.4.4]{BL}, even with $s=1/p$. 
For $0<p<2$, this follows from a slight modification of \cite[Lemma 6.2.1(1)]{BL}.)
On the other hand, it may be useful to note that
$$
f\in B_{p,p}^{1/p}(\bbR) \Rightarrow f\in W_p^{1/p}(\bbR), \quad 0<p\leq1.
$$
(Again, this follows from an adaptation of \cite[Lemma 6.2.1(1)]{BL} to $0<p\leq1$.)
In particular, $B_{1,1}^1(\bbR)\subset C(\bbR)$.

\subsection{Schatten classes}
For $0<p<\infty$, the Schatten class $\Sch_p$ is the class of all compact operators $A$
in a given Hilbert space such that 
$$
\norm{A}_{p}=\bigl(\sum_{n=1}^\infty s_n(A)^p\bigr)^{1/p}<\infty, 
$$
where $\{s_n(A)\}_{n=1}^\infty$ is the sequence of all singular values of $A$, enumerated
with multiplicities taken into account. 
The expression $\norm{\cdot}_p$ is a norm for $p\geq1$ and a quasinorm for $0<p<1$. 
For $0<p\leq1$ we have the following modified triangle inequality in $\Sch_p$: 
\[
\norm{A+B}_p^p\leq \norm{A}_p^p+\norm{B}_p^p, \quad A, B\in \Sch_p, \quad 0<p\leq1.
\label{b11}
\]
We will also need the following H\"older inequality in Schatten classes:
\[
\norm{AB}_p\leq \norm{A}_q\norm{B}_r, \quad 
\tfrac1p=\tfrac1q+\tfrac1r. 
\label{b12}
\]

\subsection{Kato smoothness}\label{sec.b4}

Here we briefly recall (with minor simplifications) the relevant 
definitions and main results of \cite{I}. 

To motivate what comes next, we should explain that we will factorise the potential $V$ in the form
$$
V=(\sign V)\abs{V}^{1-\theta}\abs{V}^\theta
$$
with an appropriate exponent $\theta\in(0,1)$. This corresponds to the ``abstract" factorisation 
$$
V=G_1^*G_0
$$
of \cite{I}. In \cite{I}, we consider the general case, where $G_0$, $G_1$  are possibly unbounded operators
from a Hilbert space $\calH$ to another Hilbert space $\calK$, such that  $G_0$ is $H_0$-bounded 
and $G_1$ is $H_1$-bounded. In this paper, since $V$ is assumed to be bounded, we will only consider the 
case of bounded operators $G_0$, $G_1$; this simplifies the exposition. 
We shall also assume $\calH=\calK$.

Let $H$ be a self-adjoint operator in a Hilbert space $\calH$
and let $G$ be a bounded operator in  $\calH$. 
One says that $G$ is \emph{Kato smooth with respect to $H$} (we will write $G\in \Smooth(H)$), if 
\begin{equation}
\norm{G}_{\Smooth(H)}:=\sup_{\norm{\varphi}_{L^2(\bbR)}=1}\norm{G\varphi(H)}<\infty.
\label{a3}
\end{equation}
As shown in \cite{I}, this definition coincides with the standard definition (see \cite{Kato1}) of Kato smoothness. 
The advantage of the definition \eqref{a3} is that it extends naturally to Schatten classes. 
Generalising \eqref{a3}, we will say that $G\in\Smooth_{p}(H)$ for some $0<p<\infty$, if 
$$
\norm{G}_{\Smooth_{p}(H)}:=\sup_{\norm{\varphi}_{L^2(\bbR)}=1}\norm{G\varphi(H)}_{p}<\infty.
$$
Finally, we shall write $G\in\Smooth_\infty(H)$, if $G\in\Smooth(H)$ and if
$$
G\1_{(-R,R)}(H)\in\Sch_\infty\quad\forall R>0.
$$
It is very easy to prove \cite[Lemma 2.3]{I} that for $G\in \Smooth_\infty(H)$, one has
\[
G\varphi(H)\in \Sch_\infty, \quad \forall \varphi\in L^2(\bbR).
\label{b10a}
\]

\subsection{Main results from \cite{I}}\label{sec.b5}
In the following theorem, $H_0$ and $H_1$ are self-adjoint operators in a Hilbert space $\calH$ 
such that the perturbation $H_1-H_0$ factorises as 
$$
H_1-H_0=G_1^*G_0,
$$
where $G_0$, $G_1$ are bounded operators in $\calH$. 
Let $\Lambda\subset \bbR$ be a measurable set; the case $\Lambda=\bbR$ is not excluded. 
(In fact, during the first reading of this subsection, the reader is encouraged to think of the simplest case $\Lambda=\bbR$.)
Here we are interested in the ``diagonal term'' in \eqref{f4}, 
$$
D_\Lambda(f):=\1_\Lambda(H_1)D(f)\1_\Lambda(H_0).
$$
Since functions $f\in\BMO(\bbR)$ in general need not be bounded, we need to take some 
care in defining the operator $D_\Lambda(f)$. We define the corresponding sesquilinear form
$$
d_{\Lambda,f}[u,v]:=
(\1_\Lambda(H_0)u,\overline{f}(H_1)\1_\Lambda(H_1)v)
-
(f(H_0)\1_\Lambda(H_0)u,\1_\Lambda(H_1)v), 
$$
for $u\in\Dom f(H_0)$, 
$v\in\Dom f(H_1)$. 
Of course, if $f$ is bounded, we can define $D_\Lambda(f)$ directly and then
\[
d_{\Lambda,f}[u,v]=(D_\Lambda(f)u,v)
\label{a4a}
\]
for all $u$ and  $v$ as above. 
We use the standard convention that if the norms in the right hand side of an upper bound
are all finite, then the bound includes the statement that the norms in the left hand side are also finite. 
The following theorem is a combination of Theorems 7.5 and 7.6 from \cite{I}. 

\begin{theorem}\label{thm.a1}
Let $H_0$, $H_1$, $G_0$, $G_1$, $\Lambda$, $d_{\Lambda,f}$ be as above. 
\begin{enumerate}[\rm (i)]
\item
For any $f\in\BMO(\bbR)$, the sesquilinear form 
$d_{\Lambda,f}[u,v]$ satisfies the bound
\begin{multline*}
\abs{d_{\Lambda,f}[u,v]}
\leq
C\norm{f}_{\BMO(\bbR)}
\norm{G_0\1_\Lambda(H_0)}_{\Smooth(H_0)}
\norm{G_1\1_\Lambda(H_1)}_{\Smooth(H_1)}
\norm{u}_\calH\norm{v}_\calH,
\end{multline*}
for any $u\in\Dom f(H_0)$, $v\in\Dom f(H_1)$, 
where the constant $C$ depends only on the choice of the norm in $\BMO(\bbR)$.
Thus, the form $d_{\Lambda,f}$ corresponds to a bounded linear operator $D_{\Lambda}(f)$ 
in $\calH$ (in the sense of \eqref{a4a}), and this operator satisfies
\[
\norm{D_{\Lambda}(f)}
\leq
C\norm{f}_{\BMO(\bbR)}
\norm{G_0\1_\Lambda(H_0)}_{\Smooth(H_0)}
\norm{G_1\1_\Lambda(H_1)}_{\Smooth(H_1)}.
\label{a4}
\]
\item
Assume that $G_0\1_\Lambda(H_0)\in\Smooth(H_0)$, $G_1\1_\Lambda(H_1)\in\Smooth(H_1)$
and at least one of the inclusions
$$
G_0\1_\Lambda(H_0)\in\Smooth_\infty(H_0), \quad G_1\1_\Lambda(H_1)\in\Smooth_\infty(H_1)
$$
holds. Then for any  $f\in \VMO(\bbR)$ the operator $D_{\Lambda}(f)$ 
is compact. 
\item
Let $p,q,r$ be finite positive indices such that $\frac1p=\frac1q+\frac1r$.
Then for any $f\in B_{p,p}^{1/p}(\bbR)\cap\BMO(\bbR)$, one has
$$
\norm{D_{\Lambda}(f)}_p
\leq 
C(p)\norm{f}_{B_{p,p}^{1/p}(\bbR)}
\norm{G_0\1_\Lambda(H_0)}_{\Smooth_{q}(H_0)}
\norm{G_1\1_\Lambda(H_1)}_{\Smooth_{r}(H_1)},
$$
where the constant $C(p)$ depends only on the choice of the function $w$ in \eqref{b9}. 
\end{enumerate}
\end{theorem}

\section{Off-diagonal terms}\label{sec.f}

Let $H_0$, $H_1$ be self-adjoint operators in $\calH$, with 
$$
H_1-H_0=G_1^*G_0=G_0^*G_1,
$$
where $G_0$ and $G_1$ are bounded operators in $\calH$. 

Let $\Lambda=(a-b,a+b)$ be a bounded open interval, and 
let $f$ be a function supported in $\Lambda$.  
In this section we estimate the norms of the off-diagonal terms in \eqref{f4}, namely,
\[
\1_{\Lambda^c}(H_1)f(H_0)\quad\text{ and }\quad f(H_1)\1_{\Lambda^c}(H_0).
\label{f1}
\]
As in the previous section, since $f$ need not be bounded, we have to take care
about defining the operators \eqref{f1}. 
We define $\1_{\Lambda^c}(H_1)f(H_0)$ initially on $\Dom f(H_0)$. 
Further, instead of $f(H_1)\1_{\Lambda^c}(H_0)$ we will consider initially 
its formal adjoint $\1_{\Lambda^c}(H_0)f(H_1)$, defined on $\Dom f(H_1)$.

The following preliminary lemma establishes a series representation for these two operators.
This representation plays the same role here as the double operator integrals in the proof
of Theorem~\ref{thm.a1} (see \cite{I}): it allows us to 
estimate the operator norms. 
Then we will refine this representation and estimate the Schatten norms in Lemma~\ref{localsmooth}.

In what follows we denote $R_0(z)=(H_0-z)^{-1}$, $R_1(z)=(H_1-z)^{-1}$.

\begin{lemma}\label{lma.f6}
Let $H_0$, $H_1$, $G_0$, $G_1$, $\Lambda$  be as described above, 
and let $f\in L^2(\bbR)$, $\supp f\subset \Lambda$. 
Assume that
$$
G_0\1_\Lambda(H_0)\in\Smooth(H_0) \quad\text{ and }\quad
G_1\1_\Lambda(H_1)\in\Smooth(H_1).
$$
Then the operator $\1_{\Lambda^c}(H_1)f(H_0)$, defined initially on $\Dom f(H_0)$, 
and the operator $\1_{\Lambda^c}(H_0)f(H_1)$, defined initially on $\Dom f(H_1)$, 
extend to bounded operators on $\calH$. 
Moreover, we have the series representations
\begin{align}
\1_{\Lambda^c}(H_0)f(H_1)
&
=
-\sum_{m=0}^\infty (H_0-a)^{-m-1}\1_{\Lambda^c}(H_0) G_0^*G_1(H_1-a)^{m}f(H_1)\,,
\label{f6}
\\
\1_{\Lambda^c}(H_1)f(H_0)
&
=
\sum_{m=0}^\infty (H_1-a)^{-m-1} \1_{\Lambda^c}(H_1) G_1^*G_0(H_0-a)^{m}f(H_0) \,,
\label{f19}
\end{align}
where both series converge absolutely in the operator norm. 
Furthermore, with $\delta=\dist(\supp f,\Lambda^c)$ and 
 $z=a+ib$, 
we have the estimates
\begin{align}
\norm{\1_{\Lambda^c}(H_0)f(H_1)}
&\leq 
\sqrt{2}
(b/\delta)\norm{f}_{L^2}\norm{G_1\1_{\Lambda}(H_1)}_{\Smooth(H_1)}\norm{G_0R_0(z)},
\label{f21}
\\
\norm{\1_{\Lambda^c}(H_1)f(H_0)}
&\leq 
\sqrt{2}
(b/\delta)\norm{f}_{L^2}\norm{G_0\1_{\Lambda}(H_0)}_{\Smooth(H_0)}\norm{G_1R_1(z)}.
\label{f22}
\end{align}
If, in addition,
$$
G_0\1_\Lambda(H_0)\in\Smooth_\infty(H_0) \quad\text{ and }\quad
G_1\1_\Lambda(H_1)\in\Smooth_\infty(H_1), 
$$
then 
$$
\1_{\Lambda^c}(H_0)f(H_1)\in \Sch_\infty 
\quad\text{ and }\quad
\1_{\Lambda^c}(H_1)f(H_0)\in \Sch_\infty.
$$
\end{lemma}
We note that although the stand-alone operator $(H_0-a)^{-m-1}$ does not necessarily make sense, 
the product $(H_0-a)^{-m-1}\1_{\Lambda^c}(H_0)$ in \eqref{f6} is well defined and bounded, because $a\in \Lambda$. 
The same comment applies to the operator $(H_1-a)^{-m-1} \1_{\Lambda^c}(H_1)$ in \eqref{f19}.
\begin{proof}
For simplicity of notation, we assume $a=0$, so $\supp f\subset [-b_0,b_0]$ with $b_0=b-\delta$. 
First observe that formally, we have
\begin{multline*}
\sum_{m=0}^\infty 
H_0^{-m-1}G_0^*G_1H_1^m
=
\sum_{m=0}^\infty 
H_0^{-m-1}(H_1-H_0)H_1^m
\\
=
\sum_{m=0}^\infty 
(H_0^{-m-1}H_1^{m+1}-H_0^{-m}H_1^{m})
=
-I.
\end{multline*}
After multiplication by  $\1_{\Lambda^c}(H_0)$ on the left and by $f(H_1)$ on the right,
we obtain \eqref{f6}. 
Now let us prove the norm convergence of the series in \eqref{f6}.
For each term, we have the estimate
\begin{multline}
\norm{\1_{\Lambda^c}(H_0)H_0^{-m-1}G_0^*G_1 H_1^m f(H_1)}
\leq
\norm{\1_{\Lambda^c}(H_0) H_0^{-m-1} G_0^*}
\norm{G_1 H_1^m f(H_1)}
\\
\leq
b^{-m} \norm{\1_{\Lambda^c}(H_0) H_0^{-1}G_0^*}
\,\, b_0^m \norm{G_1 f(H_1)}
\\
\leq
(b_0/b)^m\norm{f}_{L^2}\norm{G_0H_0^{-1}\1_{\Lambda^c}(H_0)} \norm{G_1\1_{\Lambda}(H_1)}_{\Smooth(H_1)}\, .
\label{f20}
\end{multline}
Since $b_0<b$, we have the norm convergence of the series in \eqref{f6}, and
$$
\sum_{m=0}^\infty b_0^m b^{-m}=1/(1-b_0/b)= b/\delta
$$
gives the factor $b/\delta$ in \eqref{f21}. Finally, 
$$
\norm{G_0H_0^{-1}\1_{\Lambda^c}(H_0)}
\leq
\norm{G_0R_0(ib)}\norm{H_0^{-1}\1_{\Lambda^c}(H_0)(H_0-ib)}
\leq
\sqrt{2}\norm{G_0R_0(ib)},
$$
since 
\[
\sup_{\abs{\lambda}>b}\abs{(\lambda-ib)/\lambda}\leq\sqrt{2}.
\label{f23}
\]
This gives the estimate \eqref{f21}. 

The identity \eqref{f19} and the estimate \eqref{f22} are considered similarly.
Finally, the compactness statement follows from the fact that by \eqref{b10a}, each term in the norm convergent 
series \eqref{f6}, \eqref{f19} is compact. 
\end{proof}

Now we come to the Schatten class estimate. 
It is not difficult to estimate the Schatten norm of the off-diagonal terms \eqref{f1} by the expressions 
similar to the right sides of \eqref{f21}, \eqref{f22} but with Schatten norms 
instead of the operator norms. 
However, in application to the Schr\"odinger operator, this is not sufficient, as 
the operators $G_1R_1(z)$, $G_0R_0(z)$ will not necessarily be in the required Schatten classes. 
The standard way to deal with this problem is to consider powers of the resolvent, i.e., to 
consider $G_1R_1(z)^m$, $G_0R_0(z)^{m}$ for sufficiently high $m$; these operators
will be in the required Schatten class. 
This is what we do below. The price to pay are the additional terms in the right sides of \eqref{f5} and \eqref{f5a}.

\begin{lemma}\label{localsmooth}
Assume the hypothesis of Lemma~\ref{lma.f6}, and let 
$p$, $q$, $r$ be positive finite exponents satisfying
$\frac1p=\frac1q+\frac1r$. 
Then for $z=a+ib$ and any integer $k\geq0$, 
\begin{align}
\norm{\1_{\Lambda^c}(H_0)f(H_1)}_p
& \leq 
C(b,\delta,p,k)
\bigl(\norm{f}_{L^2}\norm{G_1\1_\Lambda(H_1)}_{\Smooth_r(H_1)}
\norm{G_0R_0(z)^{k+1}}_{q}  \notag \\
& \qquad\qquad\qquad  + \left\|  \left( R_1(z)^k - R_0(z)^k \right) f(H_1)\right\|_p \bigr),
\label{f5}
\\
\norm{\1_{\Lambda^c}(H_1)f(H_0)}_p
& \leq 
C(b,\delta,p,k)
\bigl( \norm{f}_{L^2}\norm{G_0\1_\Lambda(H_0)}_{\Smooth_q(H_0)} 
\norm{G_1R_1(z)^{k+1}}_{r}
\notag\\
& \qquad\qquad\qquad + \left\| \left( R_1(z)^k - R_0(z)^k \right) f(H_0) \right\|_p \bigr).
\label{f5a}
\end{align}
\end{lemma}

\begin{proof}
For simplicity of notation, we assume $a=0$ and let $\supp f\subset [-b_0,b_0]$, $b_0=b-\delta$. 
We will prove the first bound \eqref{f5}; the second bound \eqref{f5a} is proved in the same way. 

\emph{Step 1. We prove the lemma for $k=0$.} 

We need to estimate the $\Sch_p$ norm of each term in the series in \eqref{f6}.
Similarly to \eqref{f20}, we have
\begin{multline*}
\norm{\1_{\Lambda^c}(H_0) H_0^{-m-1} G_0^*G_1    H_1^m f(H_1)}_p
\leq
\norm{G_0H_0^{-m-1}\1_{\Lambda^c}(H_0)}_{q}
\norm{G_1 H_1^m f(H_1)}_{r}
\\
\leq
b^{-m}
\norm{G_0H_0^{-1}\1_{\Lambda^c}(H_0)}_{q}
b_0^m \norm{G_1f(H_1)}_{r}
\\
\leq
(b_0/b)^m\norm{f}_{L^2}
\norm{(H_0-ib) H_0^{-1}\1_{\Lambda^c}(H_0)}
\norm{G_0R_0(ib)}_q
\norm{G_1\1_{\Lambda}(H_1)}_{\Smooth_r(H_1)}
\\
\leq
\sqrt{2}(b_0/b)^{m}\norm{f}_{L^2}
\norm{G_0R_0(ib)}_q
\norm{G_1\1_{\Lambda}(H_1)}_{\Smooth_r(H_1)},
\end{multline*}
where the last estimate uses \eqref{f23}. 
For $p\geq1$, this yields
\begin{multline*}
\norm{\1_{\Lambda^c}(H_0)f(H_1)}_p
\leq 
\sum_{m=0}^\infty \sqrt{2}(b_0/b)^{m}
\norm{f}_{L^2}
\norm{G_0R_0(ib)}_q
\norm{G_1\1_{\Lambda}(H_1)}_{\Smooth_r(H_1)}
\\
=
\sqrt{2}(b/\delta)
\norm{f}_{L^2}
\norm{G_0R_0(ib)}_q
\norm{G_1\1_{\Lambda}(H_1)}_{\Smooth_r(H_1)}. 
\end{multline*}
For $0<p<1$ we use the modified triangle inequality \eqref{b11} in $\Sch_p$, 
which yields the same estimate with a different constant. 
Thus we get the required estimate for $k=0$.

\emph{Step 2. We now consider $k>0$.}  Let $g(\lambda) = (\lambda-z)^k f(\lambda)$, so that
$$
f(H_1) = R_0(z)^k  g(H_1) + \left( R_1(z)^k - R_0(z)^k \right) g(H_1) 
$$
and therefore
\begin{align}
\label{eq:seconddecompalt}
 \1_{\Lambda^c}(H_0) f(H_1)
 & = \1_{\Lambda^c}(H_0)R_0(z)^k g(H_1)+  \1_{\Lambda^c}(H_0)\left( R_1(z)^k - R_0(z)^k \right) g(H_1)\,.
\end{align}
Let us discuss the two terms on the right side of \eqref{eq:seconddecompalt} separately.

The first term
can be estimated by the same technique as in Step 1. This yields
\begin{multline*}
\left\| \1_{\Lambda^c}(H_0)R_0(z)^k g(H_1)\right\|_p 
\leq 
C(b,\delta,p) \left\|G_1  g(H_1)\right\|_{r} \left\| G_0 H_0^{-1} R_0(z)^k \1_{\Lambda^c}(H_0) \right\|_{q} 
\\
\leq
C(b,\delta,p,k) 
\norm{f}_{L^2}\norm{G_1\1_\Lambda(H_1)}_{\Smooth_r(H_1)}
 \left\| G_0 R_0(z)^{k+1}  \right\|_{q}.
\end{multline*}
The second term in \eqref{eq:seconddecompalt} is simply estimated by
$$
\left\| \1_{\Lambda^c}(H_0) \left( R_1(z)^k - R_0(z)^k \right)  g(H_1) \right\|_p 
\leq 
2^{k/2} b^k \left\|  \left( R_1(z)^k - R_0(z)^k \right) f(H_1)\right\|_p \,.
$$
This  completes the proof of the lemma.
\end{proof}

\section{$\Sch_p$-valued smoothness for the Schr\"odinger operator}\label{sec.g}

In this section $H_0$, $H_1$ are as in \eqref{a00}. 
We set $\jap{x}=\sqrt{1+\abs{x}^2}$ and 
assume that $V(x)$ is real-valued and satisfies the condition
\begin{equation}
\abs{V(x)}\leq C\jap{x}^{-\rho}, \quad \rho>1.
\label{g1}
\end{equation}
As in Section~\ref{sec.f}, we denote the resolvents by $R_0(z)=(H_0-z)^{-1}$, $R_1(z)=(H_1-z)^{-1}$. 

\subsection{The LAP and its consequences}
First we recall the limiting absorption principle (LAP) for the Schr\"odinger operator
and translate it into statements about $\Sch_p$-valued smoothness.

\begin{lemma}\label{lma.g1}
Let $H_0$, $H_1$ be as above, with some $\rho>1$. 
Then for any $\lambda>0$, the limits
\[
\jap{x}^{-\rho/2}R_0(\lambda\pm i0)\jap{x}^{-\rho/2},
\quad 
\jap{x}^{-\rho/2}R_1(\lambda\pm i0)\jap{x}^{-\rho/2}
\label{g6}
\]
exist in the operator norm and are continuous (in the operator norm) 
in $\lambda>0$. 
Further, for any $p\geq1$, $p>\frac{d-1}{\rho-1}$, we have the inclusions
\begin{align}
\Im(\jap{x}^{-\rho/2}R_0(\lambda+i0)\jap{x}^{-\rho/2})&\in \Sch_p,
\label{g7}
\\
\Im(\jap{x}^{-\rho/2}R_1(\lambda+i0)\jap{x}^{-\rho/2})&\in \Sch_p,
\label{g8}
\end{align}
and these operators are continuous in $\lambda>0$ in $\Sch_p$. 
Finally, for the same range of $p$ we have the inclusions
\[
\jap{x}^{-\rho/2} \1_\Lambda(H_0) \in\Smooth_{2p}(H_0),
\qquad
\jap{x}^{-\rho/2} \1_\Lambda(H_1) \in\Smooth_{2p}(H_1)
\label{g9}
\]
for any bounded interval 
$\Lambda\subset\bbR$ with $\clos(\Lambda)\subset(0,\infty)$. 
\end{lemma}

\begin{proof}
The existence and continuity of the limits \eqref{g6} is the standard 
LAP, see e.g. \cite[Proposition~1.7.1, Theorem~6.2.1]{Yafaev2}. 
The inclusion \eqref{g7} and the corresponding continuity in $\lambda>0$
is also well-known; see e.g. \cite[Lemma~8.1.2]{Yafaev2}.

In order to deal with the operator in \eqref{g8}, we need a version of the resolvent identity. 
For $\Im z>0$, we have 
$$
R_1(z)=(I+R_0(z)V)^{-1}R_0(z), 
\quad
(I+R_0(z)V)^{-1}=I-R_1(z)V. 
$$
Taking the imaginary part in the first identity here and subsequently using the second identity, we obtain 
\begin{align}
\Im R_1(z)
& =
(I+R_0(z)V)^{-1}(\Im R_0(z)) (I+VR_0(z)^*)^{-1}
\notag \\
& =
(I-R_1(z)V)(\Im R_0(z))(I-VR_1(z)^*). 
\label{g2}
\end{align}
Let us denote for brevity 
$$
W(x)=\jap{x}^{-\rho/2}, \quad 
V_1(x)=V(x)\jap{x}^{\rho/2}. 
$$
Multiplying \eqref{g2} by $W$ both on the
right and on the left, we obtain
\begin{align}
\Im (WR_1(z)W)
& =
W(I-R_1(z)V)(\Im R_0(z))(I-VR_1(z)^*)W
\notag \\
& =
(I-WR_1(z)V_1)\Im(W R_0(z)W)(I-V_1R_1(z)^*W). 
\label{g3}
\end{align}
Now observe that $\abs{V_1(x)}\leq C\jap{x}^{-\rho/2}$,  and so, by the LAP \eqref{g6}, 
we can pass to the limit in the operator norm on both sides of \eqref{g3} as $z\to\lambda+i0$, $\lambda>0$. 
By \eqref{g6} and \eqref{g7}, this yields the inclusion \eqref{g8} and the continuity in $\lambda>0$.

Let us prove the first inclusion in \eqref{g9}. 
By the LAP, 
for any $\varphi\in L^2(\bbR)$, $\supp \varphi\subset\Lambda$, we have
$$
W\varphi(H_0)(W\varphi(H_0))^*
=
W\abs{\varphi(H_0)}^2W
=
\frac1\pi\int_\Lambda \abs{\varphi(\lambda)}^2 \Im (WR_0(\lambda+i0)W) d\lambda,
$$
and therefore, by \eqref{g7}, 
$$
\norm{W\varphi(H_0)}_{2p}^2=\norm{W\abs{\varphi(H_0)}^2W}_{p}
\leq
\frac1\pi
\sup_{\lambda\in\Lambda}\norm{\Im WR_0(\lambda+i0)W}_p
\int_\Lambda\abs{\varphi(\lambda)}^2 d\lambda.
$$
This gives the inclusion  $W\1_\Lambda(H_0) \in\Smooth_{2p}(H_0)$. 
The second inclusion in \eqref{g9} follows from \eqref{g8} in the same way. 
\end{proof}

\subsection{Estimates for $g(x)h(-i\nabla)$ and their consequences}
 Let us we recall two estimates for operators of the form
\[
g(x)h(-i\nabla) \quad \text{ in } L^2(\bbR^d),
\label{g11}
\]
where $g$, $h$ are complex-valued functions on $\bbR^d$ of the class to be
specified below. 
Notation \eqref{g11} is a common shorthand for operators defined by 
$$
\varphi\mapsto g(x)(\widecheck{h\widehat\varphi})(x), \quad x\in\bbR^d,\quad \varphi\in L^2(\bbR^d),
$$
where $\varphi\mapsto \widehat \varphi$ is the standard (unitary) Fourier transform and 
$\varphi\mapsto \widecheck\varphi$ is the inverse Fourier transform. See e.g. \cite[Chapter 4]{Si}
for the details. 
For $q>0$ and a complex-valued function $g$ on $\bbR^d$, we will use the notation 
$$
\norm{g}_{\ell^q(L^2)}^q:=\sum_{k\in \bbZ^d}
\biggl(\int_{(0,1)^d+k} \abs{g(x)}^2 dx\biggr)^{q/2};
\label{lattice_norm}
$$
the space $\ell^q(L^2)$ is the set of functions $g$ with $\norm{g}_{\ell^q(L^2)}<\infty$. 
\begin{proposition}\label{prp.g1}
\begin{enumerate}[\rm (i)]
\item
Let $2\leq q<\infty$ and $g,h\in L^q(\bbR^d)$. 
Then $g(x)h(-i\nabla)\in \Sch_q$ and
$$
\norm{g(x)h(-i\nabla)}_q
\leq
C_{d,q}\norm{g}_{L^q}\norm{h}_{L^q}. 
$$
\item
Let $0<q\leq 2$ and $g,h\in \ell^q(L^2)$. 
Then $g(x)h(-i\nabla)\in \Sch_q$ and
$$
\norm{g(x)h(-i\nabla)}_q
\leq
C_{d,q}\norm{g}_{\ell^q(L^2)}\norm{h}_{\ell^q(L^2)}. 
$$
\end{enumerate}
\end{proposition}
Part (i) is the Kato-Seiler-Simon inequality, see \cite{SS} or  \cite[Thm. 4.1]{Si}; 
part (ii) is the Birman-Solomyak inequality, see \cite[Thm. 11.1]{BiSo} (or \cite[Thm. 4.5]{Si} for $1\leq q\leq 2$). 
Part (ii) is used in the next lemma, and part (i) is used in the following Section.

\begin{lemma}\label{lma.g2}
Let $\sigma>0$ and $d/\sigma<q\leq 2$. 
Then $\jap{x}^{-\sigma}\1_\Lambda(H_0)\in \Smooth_{q}(H_0)$ for any bounded interval $\Lambda\subset\bbR$ 
with $\clos (\Lambda)\subset (0,\infty)$.
\end{lemma}

\begin{proof}
By Proposition~\ref{prp.g1}(ii), we have
$$
\norm{\jap{x}^{-\sigma}\1_\Lambda(H_0)\varphi(H_0)}_{q}
\leq
C
\norm{\jap{x}^{-\sigma}}_{\ell^{q}(L^2)}
\norm{\varphi(\abs{\xi}^2)}_{\ell^{q}(L^2)}. 
$$
As $\Lambda$ is bounded, the support of the function $\varphi(\abs{\xi}^2)$ in $\bbR^d$ is also bounded. 
It follows that the sum \eqref{lattice_norm} in the expression for the norm 
$\norm{\varphi(\abs{\xi}^2)}_{\ell^{q}(L^2)}$ contains only finitely many terms. 
From here it easily follows that 
$$
\norm{\varphi(\abs{\xi}^2)}_{\ell^{q}(L^2)}
\leq 
C_\Lambda \norm{\varphi}_{L^2},\qquad \supp\varphi\subset\clos\Lambda, 
$$
which completes the proof.
\end{proof}

\section{Global $\Sch_p$ conditions}\label{sec.5}
Here $H_0$, $H_1$, $V$ are as in the previous section. 
\begin{lemma}\label{lma.g3}
Let $\sigma>0$, $q>0$, $m\in\bbN$ be such that 
$$
\sigma q>d 
\qquad\text{and}\qquad
2mq>d \,.
$$
Then for $\Im z\not=0$, we have the inclusion
$\jap{x}^{-\sigma} R_0(z)^m \in\Sch_{q}$. 
Further, if $f\in\BMO(\bbR)$ has compact support in $(0,\infty)$, then  also
$\jap{x}^{-\sigma}R_0(z)^m f(H_0)\in\Sch_{q}$.
\end{lemma}

\begin{proof}
For $q\geq 2$ we use  Proposition~\ref{prp.g1}(i):
$$
\Norm{\jap{x}^{-\sigma} R_0(z)^{m}}_{q}^{q} 
\leq 
C_{q,d}
\norm{\jap{x}^{-\sigma}}_{L^q}^{q} \, \norm{ (|\xi|^2 -z)^{-m} }_{L^q}^{q} \,.
$$
This proves the first assertion since $\norm{\jap{x}^{-\sigma}}_{L^q}<\infty$ if $\sigma q >d$ 
and $\norm{ (|\xi|^2 -z)^{-m} }_{L^q}<\infty$ if $2mq>d$. 

For $0<q<2$ we use Proposition~\ref{prp.g1}(ii):
$$
\norm{ \jap{x}^{-\sigma} R_0(z)^{-m}}_{q}^{q} 
\leq 
C_{d,q} 
\norm{\jap{x}^{-\sigma}}_{\ell^{q}(L^2)}^{q} \, \norm{ (|\xi|^2 -z)^{-m} }_{\ell^{q}(L^2)}^{q}\, .
$$
Again, we have $\norm{\jap{x}^{-\sigma}}_{\ell^{q}(L^2)}<\infty$ if $\sigma q >d$ 
and $\norm{ (|\xi|^2 -z)^{-1}}_{\ell^{q}(L^2)}<\infty$ if $2mq>d$.

The assertion with an additional term in $\BMO$ follows in the same way 
since the $L^{q}$ or $\ell^{q}(L^2)$ norm of $(|\xi|^2 -z)^{-1} f(|\xi|^2)$ is still finite if $2mq>d$. 
\end{proof}

We also need an analogue of Lemma~\ref{lma.g3} with $R_1^m$ instead of $R_0^m$. 
In order to prove it, we need to consider the difference $R_1^m-R_0^m$. 
The following lemma is essentially contained in \cite{Ya74}. We include its proof for the sake of completeness.

\begin{lemma}\label{lma.g5}
Let $V$ satisfy \eqref{g1} with some $\rho>0$, let $r>0$ and let $m\geq0$ be an integer such that
$$
\rho r>d
\qquad\text{and}\qquad
2(m+1)r>d \,.
$$
Then for $\Im z\not=0$ we have the inclusion $R_1(z)^m-R_0(z)^m \in\Sch_{r}$, 
and, if $f\in\BMO(\bbR)$ has compact support, then also $f(H_0)(R_1(z)^m-R_0(z)^m)\in\Sch_{r}$.
\end{lemma}

\begin{proof}
Throughout the proof, we suppress the dependence on $z$, writing $R_0=R_0(z)$ and $R_1=R_1(z)$. 
We use induction on $m$. For $m=0$ the statement is trivial. 
Now let $m\geq 1$ and assume the claim has already been proved for all smaller values of $m$. 
We have
\begin{multline*}
R_1^m - R_0^m 
=
\sum_{l=1}^m R_1^{l-1}(R_1-R_0)R_0^{m-l}
=
-\sum_{l=1}^m R_1^{l}VR_0^{m-l+1}
\\
=
-\biggl(\sum_{l=1}^m R_0^{l}VR_0^{m-l+1}+\sum_{l=1}^m (R_1^{l}-R_0^l)VR_0^{m-l+1}\biggr). 
\end{multline*}
Separating the $l=m$ term in the second sum on the right, combining it with the left hand side
and inverting $I+VR_0$ (the inverse exists and is bounded since $\Im z\not=0$)
we obtain 
\[
R_1^m - R_0^m 
=
-\biggl(\sum_{l=1}^m R_0^{l}VR_0^{m-l+1}+\sum_{l=1}^{m-1} (R_1^{l}-R_0^l)VR_0^{m-l+1}\biggr)(I+VR_0)^{-1}. 
\label{g10}
\]
Let us consider the first sum in the right hand side here. 
Let us check the inclusions 
\[
R_0^l V R_0^{m-l+1}\in\Sch_r
\label{g4}
\]
for each $1\leq l\leq m$. We write
$$
R_0^l V R_0^{m-l+1}=\bigl(R_0^l \abs{V}^\alpha\sign(V)\bigr)\bigl(\abs{V}^\beta R_0^{m-l+1}\bigr)
$$
with $\alpha=\frac{l}{m+1}$, $\beta=\frac{m-l+1}{m+1}$. 
Setting $r_1 = r(m+1)/l$ and $r_2= r(m+1)/(m-l+1)$, and using Lemma~\ref{lma.g3}, we obtain
$$
R_0^l \abs{V}^\alpha\in\Sch_{r_1}, \quad 
\abs{V}^\beta R_0^{m-l+1}\in\Sch_{r_2}. 
$$
Now \eqref{g4} follows by application of the H\"older inequality in trace ideals \eqref{b12}.

Next, we consider the second sum in \eqref{g10}. Let us show the inclusion
\[
(R_1^l - R_0^l) V R_0^{m-l+1}\in\Sch_{r(m+1)/(m+2)}\subset\Sch_r
\label{g5}
\]
for each $1\leq l\leq m-1$.  
Let $r_1= r(m+1)/(l+1)$ and $r_2=r(m+1)/(m-l+1)$. 
Then $r_1\geq r$ and therefore $\rho r_1>d$. Moreover,
$$
2(l+1) r_1 = 2(m+1)r>d \,.
$$
Therefore, by induction hypothesis, $R_1^l-R_0^l\in\Sch_{r_1}$. On the other hand, $r_2\geq r$ and therefore $\rho r_2>d$. Moreover,
$$
2(m-l+1)r_2 = 2(m+1) r>d \,.
$$
Therefore, by Lemma \ref{lma.g3}, $V R_0^{m-l+1}\in\Sch_{r_2}$. 
By H\"older's inequality in trace ideals, since $r_1^{-1} + r_2^{-1} = ((m+2)/(m+1)) r^{-1}$, we obtain the inclusion \eqref{g5}. 
Thus, the right hand side in \eqref{g10} is in $\Sch_r$; we have completed the induction argument 
and thereby proved the first claim of the lemma. 

The second claim is proven in the same way: one checks without difficulty that \eqref{g4}, \eqref{g5} hold true
(for the same reasons as above) with an extra $f(H_0)$ term on the left. 
\end{proof}

\begin{lemma}\label{lma.g4}
Let $\sigma>0$, $q>0$, $m\in\bbN$ be such that 
$$
\rho q>d, \qquad
\sigma q>d 
\qquad\text{and}\qquad
2mq>d \,.
$$
Then for $\Im z\not=0$, we have the inclusion
$\jap{x}^{-\sigma} R_1(z)^m \in\Sch_{q}$. 
\end{lemma}
\begin{proof}
We write
$$
\jap{x}^{-\sigma} R_1(z)^m 
= 
\jap{x}^{-\sigma} R_0(z)^m + \jap{x}^{-\sigma} \left( R_1(z)^m - R_1(z)^m \right). 
$$
According to Lemma \ref{lma.g3}, the first term is in $\Sch_{q}$. 
The second term is in $\Sch_{q}$ by Lemma \ref{lma.g5} (with $r=q$). 
\end{proof}

\section{Putting it all together}\label{sec.h}

\begin{proof}[Proof of Theorem \ref{thm.a3alt}]
Throughout the proof, we set
\[
V=G_1^*G_0, \quad G_0=\abs{V}^{1/2}, \quad G_1=\sign(V)\abs{V}^{1/2},
\label{h0}
\]
and let $\Lambda\subset\bbR$ be a bounded open interval such that $\supp f\subset \Lambda$
and the closure of $\Lambda$ is contained in $(0,\infty)$. 
We consider the three terms in the right hand side of the decomposition \eqref{f4}. 

First, consider the diagonal term
\[
\1_\Lambda(H_1)D(f)\1_\Lambda(H_0).
\label{h1}
\]
By Lemma~\ref{lma.g1}, we have
$$
G_0\1_\Lambda(H_0)\in\Smooth_\infty(H_0)
\quad\text{ and }\quad 
G_1\1_{\Lambda}(H_1)\in\Smooth_\infty(H_1). 
$$
Now we can use Theorem~\ref{thm.a1}, which ensures that 
for $f\in \BMO(\bbR)$ the product \eqref{h1} is bounded, and 
for $f\in\VMO(\bbR)$ it is compact. 

Next, the off-diagonal terms 
$$
\1_{\Lambda^c}(H_1)f(H_0), \quad 
f(H_1)\1_{\Lambda^c}(H_0)
$$
are compact by Lemma~\ref{lma.f6}. 
\end{proof}

\begin{proof}[Proof of Theorem \ref{thm.a3}]
Again, we decompose $f(H_1)-f(H_0)$ as in \eqref{f4} and treat the three terms separately.
Instead of following the cases (i) and (ii) as in the statement of the theorem, it will be convenient
to split the range of variables as follows: $p\geq1$ and $0<p<1$. 

\textbf{Case $p\geq1$.} 
Throughout the consideration of this case we use the factorisation \eqref{h0}. 
Observe that for $p\geq1$ both in case (i) and in case (ii) we have 
$$
p>\frac{d}{\rho}\quad \text{ and } \quad p>\frac{d-1}{\rho-1}. 
$$

\emph{The diagonal term.} We use Theorem~\ref{thm.a1}(iii) and take $q=r=2p$. 
Both terms $\norm{G_0\1_{\Lambda}(H_0)}_{\Smooth_{2p}(H_0)}$ and 
$\norm{G_1\1_{\Lambda}(H_1)}_{\Smooth_{2p}(H_1)}$ are finite as shown in Lemma~\ref{lma.g1}.

\emph{The  term $\1_{\Lambda^c}(H_1)f(H_0)$.} 
Let $k\geq0$ be an integer sufficiently large such that $2(k+1)p>d$. 
We use the bound \eqref{f5a} from Lemma \ref{localsmooth}. 
As already mentioned, the norm $\norm{G_0\1_{\Lambda}(H_0)}_{\Smooth_{2p}(H_0)}$ is finite. 
Moreover, according to Lemma \ref{lma.g4}, the assumptions $\rho p>d$ and $4(k+1)p>d$ 
imply that $G_1R_1(z)^{k+1}\in\Sch_{2p}$ for $\Im z\not=0$. 
If $k=0$, this already shows that $\1_{\Lambda^c}(H_1)f(H_0)\in\Sch_p$.

If $k\geq1$, we still need to show that $(R_0(z)^k - R_1(z)^k) f(H_0)\in\Sch_p$. 
This follows from Lemma \ref{lma.g5} (by taking adjoints).

\emph{The term $f(H_1)\1_{\Lambda^c}(H_0)$.} 
The argument in this case is similar to that for the second term and we will be brief. 
We choose $k$ as before and this time, we use bound \eqref{f5} from Lemma \ref{localsmooth}. 
We already know that $G_1\1_{\Lambda}(H_1)\in\Smooth_{2p}(H_1)$ and we infer that 
$G_0 R_0(z)^{k+1}\in\Sch_{2p}$ from Lemma \ref{lma.g3}.
This concludes the proof for $k=0$. 

For $k\geq1$, we still need to show that $f(H_1)(R_1(z)^k - R_0(z)^k)\in\Sch_p$. We write
$$
f(H_1)(R_1(z)^k - R_0(z)^k) = f(H_0)\left(R_1(z)^k - R_0(z)^k\right) + D(f) \left(R_1(z)^k - R_0(z)^k\right).
$$
Since $f$ is compactly supported and $f\in B^{1/p}_{p,p}$, we have $f\in \BMO$ and therefore the operator $D(f)$ is bounded by Theorem \ref{thm.a3alt}. Thus, it suffices to prove that
$$
f(H_0)\left(R_1(z)^k - R_0(z)^k\right) \,,\ R_1(z)^k - R_0(z)^k \in\Sch_p \,.
$$
This is again a consequence of Lemma \ref{lma.g5}.

\medskip

\textbf{Case $0<p<1$.} 
Here we are in the setting of part (ii) where $\rho>d$. Again, we treat separately the three terms in \eqref{f4}. 
This time we 
split the perturbation $V=G_1^* G_0$ with
$$
G_0 = (\sgn V) |V|^\theta
\qquad\text{and}\qquad
G_1 = |V|^{1-\theta} \,.
$$
Here $0<\theta<1$ is chosen such that, with $q=2p/(2-p)$,  
we have $\theta\rho q>d$ and $(1-\theta)\rho>d/2$. (Such choice of $\theta$ is possible since $p>d/\rho$.)

\emph{The diagonal term.} We use Theorem~\ref{thm.a1}(iii) with $q=2p/(2-p)$ and $r=2$. 
The term $\norm{G_0\1_{\Lambda}(H_0)}_{\Smooth_q(H_0)}$ is finite by Lemma~\ref{lma.g2} since $\theta\rho q>d$. 
Let us check that the term $\norm{G_1\1_{\Lambda}(H_1)}_{\Smooth_2(H_1)}$ is finite.

Let $\widetilde\rho = \min\{\rho,2(1-\theta)\rho\}$. 
Then $V$ satisfies \eqref{g1} with $\widetilde\rho$ instead of $\rho$. 
Moreover, $\widetilde\rho>1$ (since $\rho>1$ and $2(1-\theta)\rho>d\geq 1$) 
and $1>(d-1)/(\widetilde\rho -1)$ (since $\rho>d$ and $2(1-\theta)\rho>d$). 
Therefore, we can apply Lemma~\ref{lma.g1} with $p=1$ and with $\widetilde\rho$ instead of $\rho$. 
This gives $\jap{x}^{-\widetilde\rho/2} \1_\Lambda(H_1)\in\Smooth_2(H_1)$. 
On the other hand, $|V|^{1-\theta}\jap{x}^{\widetilde\rho/2}$ is bounded 
and therefore $G_1 \1_\Lambda(H_1) \in \Smooth_2(H_1)$.

\emph{The  term $\1_{\Lambda^c}(H_1)f(H_0)$.} 
Let $k\geq0$ be an integer sufficiently large so that  $2(k+1)p>d$.
We use bound \eqref{f5a} with the exponents $q=2p/(2-p)$, $r=2$. 
We already know that $G_0\1_{\Lambda}(H_0)\in \Smooth_q(H_0)$. 
Further, according to Lemma \ref{lma.g4}, the assumptions $(1-\theta)\rho >d/2$ and $4(k+1)>d$ 
imply that $G_1R_1(z)^{k+1}\in\Sch_{2}$ for $\Im z\not=0$.
If $k=0$, this already shows that $\1_{\Lambda^c}(H_1)f(H_0)\in\Sch_p$.

If $k\geq1$, we argue as in the case $p\geq 1$ that $(R_0(z)^k - R_1(z)^k) f(H_0)\in\Sch_p$.

\emph{The term $f(H_1)\1_{\Lambda^c}(H_0)$.} 
Again, the argument is similar and we will be brief. 
We choose $k$ as before and this time, we use bound \eqref{f5}. 
We already know that $G_1\1_{\Lambda}(H_1)\in\Smooth_2(H_1)$, and 
we infer that $G_0 R_0(z)^{k+1}\in\Sch_{q}$ from Lemma~\ref{lma.g3} since $\theta\rho q>d$ and $2(k+1)q>d$. 
If $k=0$, this already shows that $f(H_1)\1_{\Lambda^c}(H_0)\in\Sch_p$.

If $k\geq1$, we argue as in the case $p\geq 1$ that $f(H_1)(R_1(z)^k - R_0(z)^k)\in\Sch_p$. 
This concludes the proof of the theorem.
\end{proof}


\end{document}